\documentclass[10pt, a4paper, reqno]{amsart}

\usepackage[T1]{fontenc}
\usepackage{lmodern}
\usepackage[english]{babel}
\usepackage{mathrsfs}
\usepackage[dvipsnames]{xcolor}
\usepackage{hyperref}
\usepackage{booktabs}
\usepackage{float}
\hypersetup{
    colorlinks,
    linkcolor={red!50!black},
    citecolor={blue!50!black},
    urlcolor={blue!80!black}
}
\usepackage[backgroundcolor=blue!10]{todonotes}
\usepackage{algorithm}
\usepackage{algpseudocode} 
\usepackage{enumitem}
\usepackage{graphicx}
\usepackage{tikz,pgf}
\usepackage{multirow}
\usepackage{array}

\theoremstyle{plain}
\newtheorem{lemma}{Lemma}[section]
\newtheorem{proposition}[lemma]{\textbf{Proposition}}
\newtheorem{theorem}[lemma]{\textbf{Theorem}}

\theoremstyle{definition}
\newtheorem{definition}[lemma]{\textbf{Definition}}

\newtheorem{remark}[lemma]{Remark}

\newcommand{\Z}{\mathbb{Z}}

\algrenewcommand\algorithmicwhile{\textbf{While}}
\algrenewcommand\algorithmicfor{\textbf{For}}
\algrenewcommand\algorithmicdo{\textbf{Do}}
\algrenewcommand\algorithmicif{\textbf{If}}
\algrenewcommand\algorithmicthen{\textbf{Then}}
\algrenewcommand\algorithmicelse{\textbf{Else}}
\algrenewcommand\algorithmicend{\textbf{End}}
\algrenewcommand\algorithmicreturn{\textbf{Return}}

\newcommand{\R}{\mathbb{R}}
\newcommand{\C}{\mathbb{C}}

\newcommand{\A}{\mathbb{A}}
\newcommand{\p}{\mathbb{P}}
\newcommand{\SP}{\mathbb{S}}

\newcommand{\range}[2]{\lbrace #1,\ldots,#2\rbrace}	%set of integers in interval

\newcommand{\mscr}{\mathscr}
\newcommand{\mcal}{\mathcal}

\newcommand{\SO}{\operatorname{SO}}
\newcommand{\pgl}{\p\!\operatorname{GL}}

\newcommand{\M}{\overline{\mscr{M}}}

\DeclareMathOperator{\Lam}{Lam}

\tikzstyle{vertex}=[circle, draw, fill=black, inner sep=0pt, minimum size=4pt]
\tikzstyle{edge}=[line width=1.5pt,black!50!white]
\tikzstyle{gridp}=[inner sep=1pt,circle,fill=black!70!white]
\tikzstyle{gridl}=[black!50!white]

\tikzstyle{lnode}=[circle,white,draw=black!60!white,fill=black!60!white,inner 
sep=1pt]
\tikzstyle{cnode}=[circle,draw=black!60!white,fill=black!60!white,inner 
sep=1.5pt]
\tikzstyle{redge}=[edge,Red]
\tikzstyle{bedge}=[edge,NavyBlue]

\colorlet{ncol}{Green!60!black}
\tikzstyle{nvertex}=[vertex, draw=ncol, fill=ncol]
\tikzstyle{edgeq}=[edge,gray!60,densely dashed]
\tikzstyle{nedge}=[edge,ncol]
\tikzstyle{oedge}=[edge,Red!60!black]

\tikzstyle{curveline}=[line width=0.6mm]
\tikzstyle{markedpoint}=[circle, draw, fill]

\newcommand{\lowerparen}[2]{%
  \raisebox{-#1}{\(\displaystyle\left(\raisebox{#1}{\(\displaystyle #2\)}\right)\)}}

\title{Counting realizations of Laman graphs on the sphere}

\author[M. Gallet]{%
Matteo Gallet$^{\ast, \circ,\diamond}$}
\author[G. Grasegger]{%
Georg Grasegger$^{\ast, \triangleright}$}
\author[J. Schicho]{%
Josef Schicho$^{\ast, \circ}$}
\thanks{$^\ast$ Supported by the Austrian Science Fund (FWF): W1214-N15, 
 Project DK9.} 
\thanks{$^\circ$ Supported by the Austrian Science Fund (FWF): P31061.}
\thanks{$^\diamond$ Supported by the Austrian Science Fund (FWF): Erwin 
Schr\"odinger Fellowship J4253.}
\thanks{$^\triangleright$ Supported by the Austrian Science Fund (FWF): P31888.}
\address[MG]{International School for Advanced Studies/Scuola Internazionale Superiore di Studi Avanzati (ISAS/SISSA),
Via Bonomea 265, 34136 Trieste, Italy}
\email{mgallet@sissa.it}
\address[JS]{Research Institute for Symbolic Computation (RISC), Johannes 
Kepler University}
\email{jschicho@risc.jku.at}
\address[GG]{Johann Radon Institute for Computation and Applied Mathematics 
(RICAM), Austrian Academy of Sciences}
\email{georg.grasegger@ricam.oeaw.ac.at}

\begin{document}

\begin{abstract}
 We present an algorithm that computes the number of realizations of a Laman 
graph on a sphere for a general choice of the angles between the vertices. The 
algorithm is based on the interpretation of such a realization as a point in 
the moduli space of stable curves of genus zero with marked points, and on the 
explicit description, due to Keel, of the Chow ring of this space.
\end{abstract}

\maketitle

\section{Introduction}

Maybe the most important open problem in \emph{rigidity theory} is the 
characterization and study of rigid structures in three dimensional space. On 
the other side, planar structures are reasonably well-understood, in the sense 
that, for example, we have a characterization of graphs that are 
\emph{generically minimally rigid} in the plane. These are the graphs that, 
once a general assignment for the edge lengths is prescribed, admit only 
finitely many ways of realizing them in the plane respecting the assignment, if 
we consider equivalent realizations that differ by an isometry. 
Pollaczek-Geiringer \cite{Geiringer1927} and Laman \cite{Laman1970} described 
these graphs in terms of their combinatorics, and so they also go under the name 
of \emph{Laman graphs}. In~\cite{Capco2018}, using ideas from tropical 
geometry, the authors provide a recursive algorithm that computes the number of 
realizations of a Laman graph for a general assignment of edge lengths, up to 
plane isometries, if one allows complex coordinates for the points of the 
realization. It has been proven by Eftekhari et al. \cite{Eftekhari2018} that 
Laman graphs are generically minimally rigid also when we consider realizations 
on the sphere, so as before one can ask in how many different ways one can 
realize a Laman graph on the sphere. In this paper, we provide a recursive 
algorithm that computes this number (again, under the assumption that complex 
coordinates for the points are allowed) based on a completely different 
technique from the one used in the planar case. We hope that, although we still 
work on a surface, moving from the plane to the sphere could be a first step 
towards determining the number of realizations for generically minimally 
rigid graphs in three dimensions. For a related work on this topic, discussing 
real realizations of graphs on the sphere (in addition to the plane and the 
space), see the recent paper by Bartzos et al.~\cite{Bartzos2018}. Among 
other things, the latter paper proves that for some graphs one can achieve all 
possible complex realizations via real instances.

\begin{figure}[ht]
  \begin{center}
    \includegraphics[height=4.5cm,trim=1.5cm 0 1cm 0,clip=true]{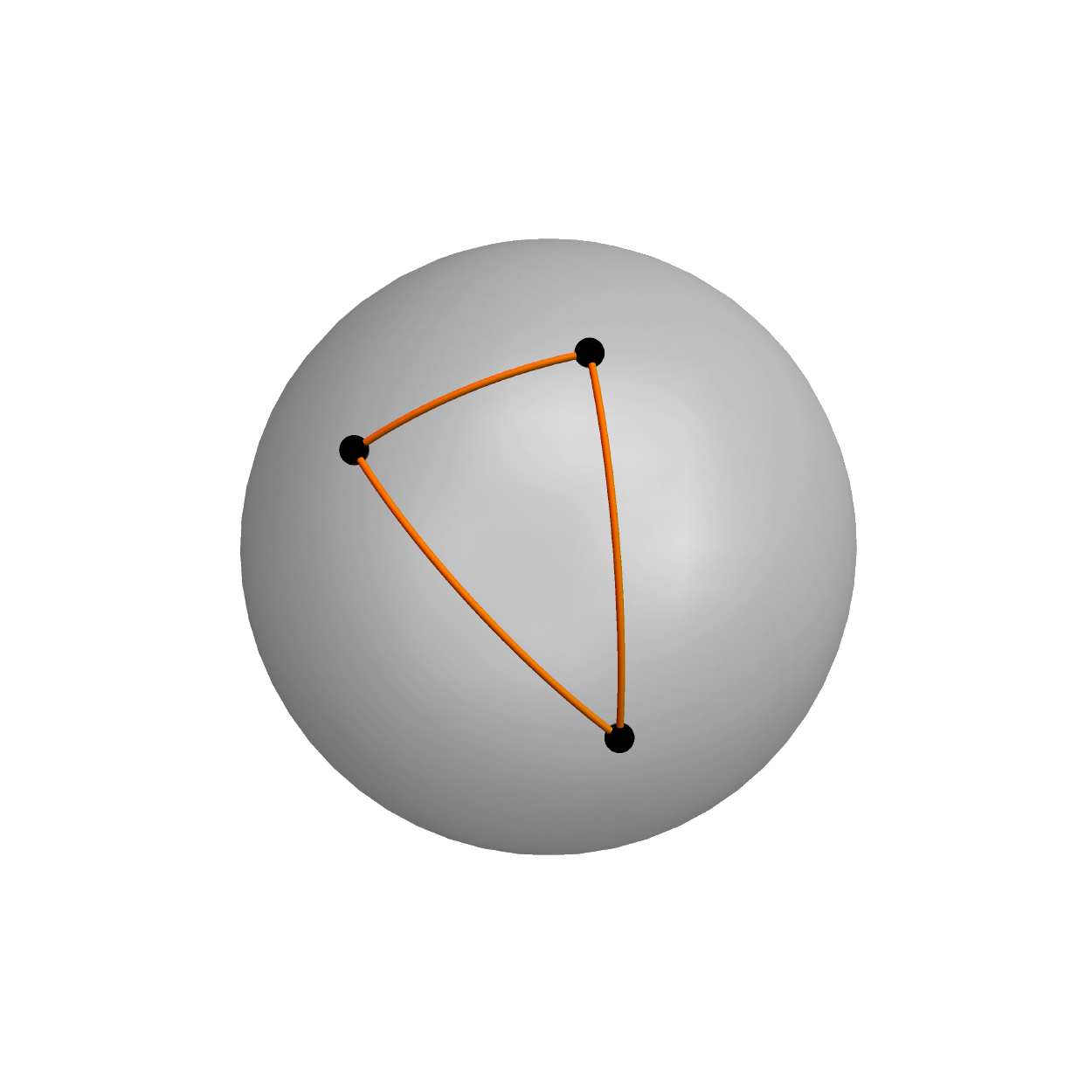}
    \includegraphics[height=4.5cm]{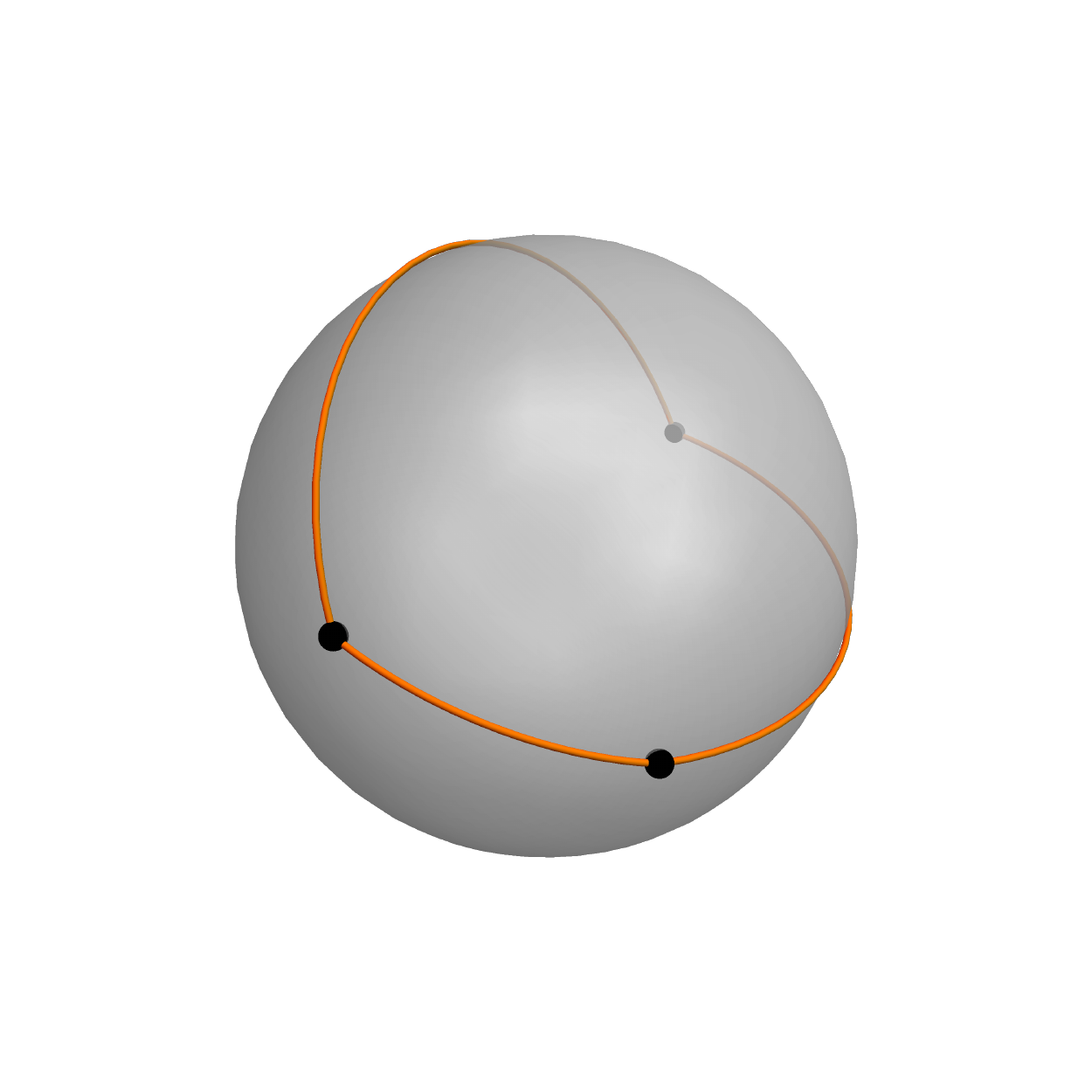}
    \includegraphics[height=4.5cm,trim=1cm 0 1.5cm 0,clip=true]{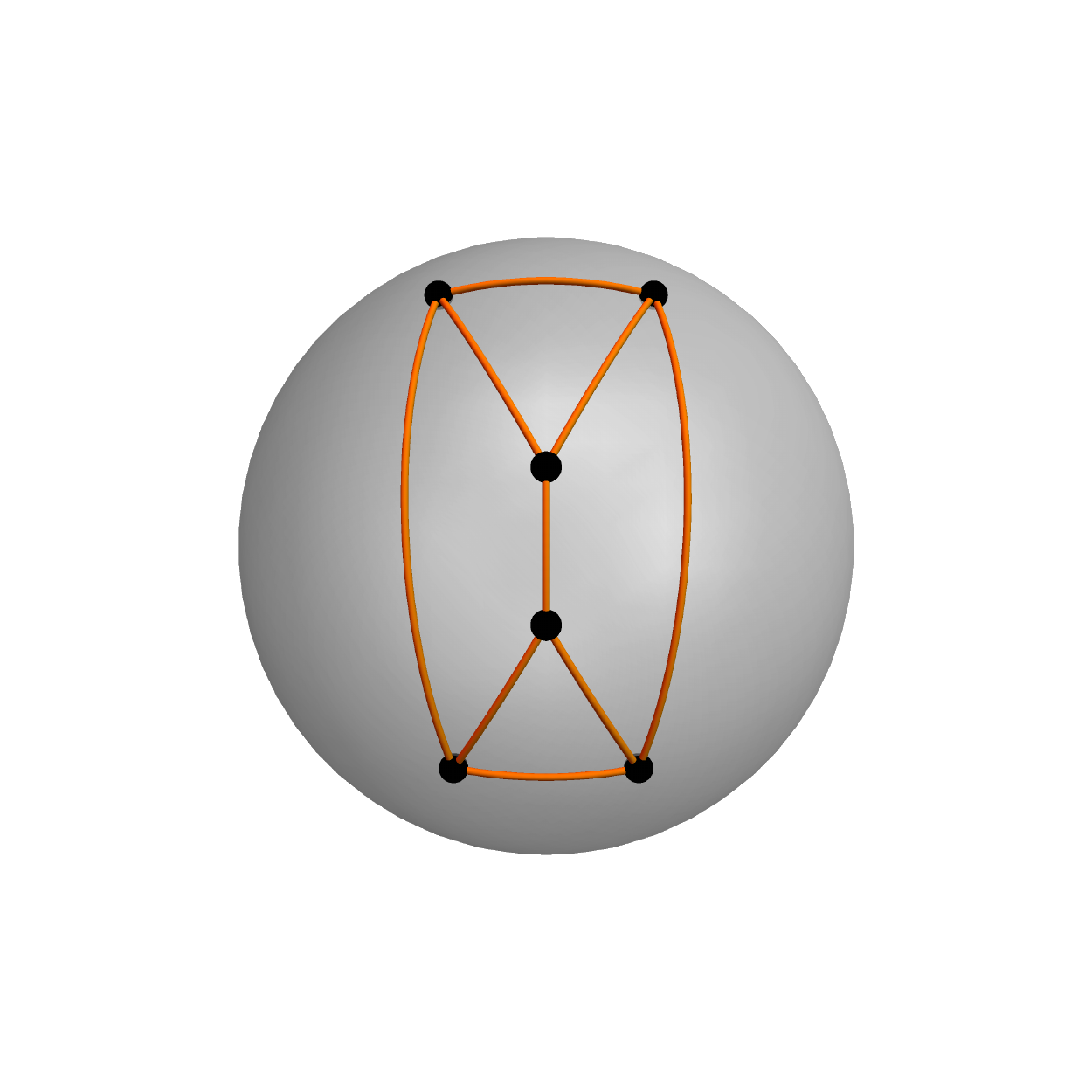}
    \caption{Realizations of graphs on the sphere.}
	\end{center}	
\end{figure}

\textbf{Our result.}
The main result of this paper is an algorithm that computes the number of 
realizations of a Laman graph on the (complex) sphere, up to the action of the 
group~$\SO_3(\C)$ (recall that the real analogue of this group, $\SO_3(\R)$, is 
the group of isometries of the real sphere). The key idea is to interpret 
realizations on the sphere up to~$\SO_3(\C)$ as elements of a moduli space, the 
so-called \emph{moduli space of rational curves with marked points}, where 
each point of the realization corresponds to two marked points. In this 
interpretation, assigning the distance between two points on the sphere 
corresponds to prescribing the cross-ratio of the $4$ related marked points. By 
using the properties of the moduli space, in particular the description of its 
Chow ring and the geometry of some of its divisors, we compute the cardinality 
of those elements for which the cross-ratios of the points corresponding to the 
edges of a Laman graph are assigned; this provides the answer to our original 
problem. Remarkably, we have been informed by Gaiane Panina that the moduli 
space of rational curves with marked points appears also in investigations of 
flexible polygons, see~\cite{Nekrasov2018}.

\textbf{Structure of the paper.}
Section~\ref{realizations} describes the problem we want to address in this 
paper. Section~\ref{realizations_as_moduli} provides the translation from 
realizations of a graph on the sphere to points of a moduli space of 
points on the projective line. Section~\ref{moduli_space} describes, 
following~\cite{Keel1992}, the moduli space of points on the projective line, 
and its compactification as moduli space of stable curves of genus zero with 
marked points; in particular, we recall the description by Keel of the Chow 
rings of the latter, which plays a key role in the algorithm. 
Section~\ref{algorithm} describes the algorithm. Section~\ref{data} reports some 
computational data obtained using the algorithm.

\textbf{Acknowledgments.} We thank the Erwin Schr\"odinger Institute (ESI) of 
the University of Vienna for the hospitality during the workshop ``Rigidity 
and Flexibility of Geometric Structures'', when this project started. We thank 
Anthony Nixon for providing us with useful references concerning rigidity 
of graphs on the sphere. We thank Jan Legersk\'y for helpful discussions on 
the topic of this paper.

\section{Realizations of graphs}
\label{realizations}

A \emph{Laman graph} is a graph $G = (V, E)$ such that $|E| = 2|V| - 3$ and 
$|E'| \leq 2|V'| - 3$ for every subgraph $G' = (V', E')$. Geiringer 
\cite{Geiringer1927} and Laman \cite{Laman1970} proved that these graphs are 
\emph{generically infinitesimally rigid} in the plane. This means the following. 
A \emph{realization} of~$G$ is a tuple $(P_v)_{v \in V}$ of points in the plane 
indexed by the vertices of~$G$. By applying a Euclidean isometry, we can always 
suppose that one of the points is the origin~$O$ of the plane, and another one 
lies on the $x$-axis. The set of all possible realizations satisfying the 
previous two requirements is then given by $\{ O \} \times \A^1_{\R} \times 
\bigl( \A^2_{\R} )^{n-2}$, where $n$ is the number of vertices. We can now 
consider the function that computes, for every edge $\{a,b\} \in E$, the 
distance $d_{\A^2_{\R}}(P_a, P_b)$ of the corresponding points in a realization. 
In this way we get a map: 
\[
 \begin{array}{rccc}
  \Psi_{\A^2_{\R}} \colon & \{ O \} \times \A^1_{\R} \times \bigl( \A^2_{\R} 
)^{n-2} & 
\longrightarrow & \R^{|E|} \\
  & (P_v)_{v \in V} & \mapsto & \bigl( d^2_{\A^2_{\R}}(P_a, P_b) 
\bigr)_{\{a,b\} \in E}
 \end{array}
\]
Notice that the map~$\Psi_{\A^2_{\R}}$ is a smooth map (better, an algebraic 
one) between smooth manifolds (better, algebraic varieties) of the same 
dimension. Laman proved that, if we pick a general point~$\vec{P}$ of the domain 
(namely, if we remove a finite number of ``bad'' subvarieties from the domain), 
then the Jacobian of~$\Psi_{\A^2_{\R}}$ at~$\vec{P}$ is invertible, i.e., the 
map~$\Psi_{\A^2_{\R}}$ is an isomorphism locally around~$P$. Notice that the 
fibers of~$\Psi_{\A^2_{\R}}$, namely the sets~$\Psi_{\A^2_{\R}}^{-1}(\lambda)$ 
for some $\lambda \in \R^{|E|}$, are the sets of realizations of~$G$ where the 
distances between points, whose corresponding vertices are connected by an edge, 
are prescribed by~$\lambda$. We give a name to these sets:

\begin{definition}
 Let $G = (V,E)$ be a Laman graph and let $\lambda \colon E \longrightarrow \R$. 
A realization of~$G$ \emph{compatible with} $\lambda$ is a function $\rho 
\colon V \longrightarrow \A^2_{\R}$ such that 
\[
  d^2_{\A^2_{\R}} \bigl( \rho(a), \rho(b) \bigr) = \lambda(\{a,b\}) 
  \quad
  \text{for every } \{a,b\} \in E \,.
\]
\end{definition}

When a fiber of~$\Psi_{\A^2_{\R}}$ is finite---namely when the number of 
realizations of~$G$ compatible with given edge lengths is finite---one may be 
interested in counting its cardinality, namely the number of ways of realizing a 
Laman graph with a specified assignment of edge lengths; however, over the real 
numbers, the cardinality of such fibers may depend on the point in the codomain. 
Since the result of Laman proves also that the map~$\Psi_{\A^2_{\R}}$ is 
dominant, if we pass to the complex numbers we have that, for a general 
element~$\lambda$ in the codomain, the fiber $\Psi_{\A^2_{\C}}^{-1}(\lambda)$ is 
finite and its cardinality does not depend on the point. In recent years, there 
has been some interest in providing lower and upper bounds for this number (see 
\cite{Borcea2004, Steffens2010, Emiris2013, Grasegger2018, Jackson2018, 
Bartzos2018}), and in~\cite{Capco2018} the authors provide an iterative formula 
to compute it, based on tropical geometry. A fully tropical proof of Laman 
theorem has recently been provided in~\cite{Bernstein2018}.

One can consider the notion of being generically infinitesimally rigid also on 
the sphere. In \cite{Eftekhari2018} it is proven that, also on the sphere, the 
class of generically infinitesimally rigid graphs coincides with the class of 
Laman graphs. On the sphere, the distance between two points can be taken to be 
the angle they form (viewed from the origin of the sphere). In this paper we 
adopt a slightly different definition, which involves the cosine of that angle, 
because it fits better in the algebraic framework we are going to use. Adopting 
this definition has no impact as the matter of computing the number of 
realizations of a graph is concerned. The advantage of this choice is that it 
provides an algebraic function, which hence allows extensions of fields (in 
particular, from the real to the complex numbers).

\begin{definition}
\label{definition:spherical_distance}
 Given two points $P, Q \in S^2$, we define their \emph{spherical distance} as 
\[
 d_{S^2}(P,Q) := \frac{1 - \langle P, Q \rangle}{2} \,,
\]
 where $\langle P, Q \rangle = \sum_{i=1}^{3} P_i Q_i$. In particular, if $P$ 
 and $Q$ are antipodal, their spherical distance is~$1$. 
\end{definition}

In this context, we can repeat the same considerations as before: given a 
configuration~$\vec{P}$ on the sphere of a Laman graph~$G = (V,E)$, we can 
always suppose, by applying rotations, that one of the points is $(1,0,0)$ and 
another lies on a great circle through~$(1,0,0)$. We then pass to the complex 
setting and have that, as in the plane, if $G$ is a Laman graph, then the map
\[
 \begin{array}{rccc}
  \Psi_{S^2_{\C}} \colon & \{ (1,0,0) \} \times S^1_{\C} \times \bigl( S^2_{\C} 
)^{n-2} & 
\longrightarrow & \C^{|E|} \\
  & (P_v)_{v \in V} & \mapsto & \bigl( d_{S^2}(P_a, P_b) \bigr)_{\{a,b\}\in E}
 \end{array}
\]
is dominant and its fibers over general points are finite and of constant 
cardinality. Here, we denoted by~$S^2_{\C}$ the set $\{ (x,y,z) \in \C^3 \, : \, 
x^2 + y^2 + z^2 = 1\}$, namely the complexification of the sphere, and similarly 
for the circle~$S^1_{\C}$. As we remarked, since the function~$d_{S^2}$ 
describing the spherical distance is algebraic, we can apply it also to pairs of 
complex points in~$S^2_{\C}$. Notice that in the real setting we consider 
realizations of the graph up to rotations, namely elements of~$\SO_3(\R)$; when 
we pass to the complex numbers, we consider realizations up to~$\SO_3(\C)$, 
where
\[
 \SO_3(\C) := 
 \bigl\{
  R \in \C^{3 \times 3} \, \colon \, R R^{t} = R^{t} R = \mathrm{id}, \, 
  \det(R) = 1
 \bigr\} \,.
\]
In this paper, the elements of~$\SO_3(\C)$ will also be (improperly) referred as 
rotations, or isometries of~$S^2_{\C}$. 

This is the main goal of our paper:

{\textbf{Goal.}} Compute the cardinality of a general fiber of the 
map~$\Psi_{S^2_{\C}}$. In other words, compute the number of realizations of a 
Laman graph on the sphere compatible with a general assignment of spherical 
distances for its edges, up to $\SO_3(\C)$.

We are going to achieve this goal by interpreting realizations 
up to~$\SO_3(\C)$ as collections of points on the projective complex 
line~$\p^1_{\C}$, up to the action of~$\pgl_2(\C)$, namely the group of 
automorphisms of~$\p^1_{\C}$. These objects can be interpreted as points in a 
moduli space, and the explicit description of the intersection theory on that 
moduli space provides the answer to our question.

\section{Realizations on the sphere as points on a moduli space}
\label{realizations_as_moduli}

The aim of this section is to show how we can interpret a realization of a 
graph on the sphere, up to sphere isometries, as a point of the moduli 
space of stable curves of genus zero with marked points. This provides the 
theoretical background on which the algorithm presented in 
Section~\ref{algorithm} is based.

We would like to express the spherical distance between two points 
in~$S^2_{\C}$ as the cross-ratio of four points in~$\p^1_\C$. To do so, we 
associate to each point in~$S^2_{\C}$ two points in~$\p^1_\C$ via the 
following construction.

\begin{definition}
Let 
\[
 \SP_{\C}^2 = \bigl\{ (x:y:z:w) \in \p^2_{\C} \,\colon\, x^2 + y^2 + z^2 - 
w^2 = 0 \bigr\}
\]
be the projective closure of~$S^2_{\C}$ in $\p^2_{\C}$. Let $A$ be the 
intersection 
of~$\SP_{\C}^2$ with the plane at infinity $\{ w = 0 \}$. The conic~$A$ is 
called the \emph{absolute conic}. Since~$\SP^2_{\C}$ is a smooth quadric 
in~$\p^3_{\C}$ there are exactly two families of lines 
on~$\SP_{\C}^2$; we denote them by~$\mscr{F}_1$ and~$\mscr{F}_2$. Every point 
$P \in \SP^2_{\C}$ is contained in exactly one line~$L_1$ 
of~$\mscr{F}_1$ and exactly one line~$L_2$ of~$\mscr{F}_2$. The union of these 
two lines can be obtained by intersecting~$\SP^2_{\C}$ with the tangent plane 
of~$\SP^2_{\C}$ at~$P$. We define the \emph{left} and the 
\emph{right lifts} of~$P$ as the intersections of $L_1$ and $L_2$ with $A$, 
respectively. We denote them by~$P^l$ and~$P^r$, respectively. 
\end{definition}

\begin{remark}
 Notice that the absolute conic~$A$ is a rational curve. This means that, given 
four points on~$A$ (for example, the left and right lifts of two points 
in~$S^2_{\C}$), we can speak about their cross-ratio.
\end{remark}

\begin{lemma}
\label{lemma:equivalence}
 Let $P, Q \in S^2_{\C}$. Let $P^l, P^r$ be the left and right lifts of~$P$, 
and $Q^l, Q^r$ be the left and right lifts of~$Q$. Then
\[
 d_{S^2}(P, Q) = 
 \frac{\mathrm{cr}\bigl( P^l, P^r, Q^l, Q^r \bigr)}{\mathrm{cr}\bigl( P^l, P^r, Q^l, Q^r \bigr) - 1} = 
 \mathrm{cr}\bigl( P^l, Q^r, Q^l, P^r \bigr),
\]
where $\mathrm{cr}$ stands for cross-ratio.
\end{lemma}
\begin{proof}
 Recall that isometries of~$S^2_{\C}$ are projective automorphisms 
of~$\p^3_{\C}$ leaving the absolute conic~$A$ invariant. Hence applying an 
isometry to~$P$ and~$Q$ corresponds to applying an automorphism of~$\p^1_{\C}$ 
to their lifts, so the cross-ratio of the latter does not change. Hence we can 
suppose that $P = (1,0,0)$ and $Q = (c,s,0)$, where $c^2 + s^2 = 1$. With this 
choice, we have
\[
 d_{S^2}(P, Q) = \frac{1 - c}{2} \,.
\]
A direct computation shows that, since the tangent planes at~$P$ and~$Q$ have 
equations $x-w=0$ and $c\,x + s\,y - w = 0$, respectively:
\begin{align*}
 P^l &= (0:i:1:0)\,, & P^r &= (0:-i:1:0)\,, \\
 Q^l &= (-is:ic:1:0)\,, & Q^r &= (is: -ic:1:0)\,.
\end{align*}
In order to compute the cross-ratio, we take an isomorphism between~$A$ 
and~$\p^1_{\C}$, for example the one given by the projection of~$A$ from the 
point~$(i:0:1:0)$ to the line $\{z = w = 0\}$. The projections of the previous 
four points are
\[
 (-1:1:0:0), \; (1:1:0:0), \; (-1-s:c:0:0), \; (1-s:c:0)\,.
\]
Their cross-ratio is
\begin{align*}
 \left(\frac{-1-s}{c} + 1\right)\left(\frac{1-s}{c} - 1\right) \bigg/
 \left(\left(\frac{-1-s}{c} - 1\right)\left(\frac{1-s}{c} + 1\right)\right) \,.
\end{align*}
A direct computation then proves the statement.
\end{proof}

\begin{proposition}
\label{proposition:translation}
 Let $\vec{P} = (P_1, \dotsc, P_n)$ and $\vec{Q} = (Q_1, \dotsc, Q_n)$ be two 
$n$-tuples of points in~$S^2_{\C}$. Denote by $P_i^{l}, P_i^{r}$ and $Q_i^{l}, 
Q_i^{r}$ the left and right lifts of $P_i$ and $Q_i$, respectively, for all $i 
\in \{1, \dotsc, n\}$. Then $\vec{P}$ and $\vec{Q}$ differ by an isometry 
of~$S^2_{\C}$ if and only if $(\vec{P^{l}}, \vec{P^{r}})$ and $(\vec{Q^{l}}, 
\vec{Q^{r}})$ differ by an element of~$\pgl_2(\C)$.
\end{proposition}
\begin{proof}
 Every isometry of~$S^2_{\C}$ is a projective automorphism of~$\p^3_{\C}$ 
leaving the absolute conic~$A$ invariant. This means that every isometry 
of~$S^2_{\C}$ determines an automorphism of~$A$. In this way we get a map
\[
 \SO_3(\C) \longrightarrow \pgl_2(\C) \,,
\]
which is a homomorphism of Lie groups. Our statement is proven if we can show 
that this is an isomorphism. Suppose that we have an isometry of~$S^2_{\C}$ 
that induces the identity on~$A$. Then the corresponding projective 
automorphism of~$\p^3_{\C}$ fixes the whole plane at infinity, and also the 
center of~$S^2_{\C}$: the only element in~$\SO_3(\C)$ satisfying these 
requirements is the identity. Hence, the homomorphism is injective. Since the 
two Lie groups have the same dimension and are both connected, then the 
homomorphism is also surjective. This concludes the proof.
\end{proof}

\begin{remark}
\label{remark:distinct}
 Notice that, in general, it is not true that $n$ distinct points on~$S^2_{\C}$ 
determine $2n$-tuples of distinct points on the absolute conic via the lift 
operation. In fact, this fails precisely when two points on~$S^2_{\C}$ belong 
to the same line. If $P,Q \in S^2_{\C}$ with $P = (\alpha,\beta,\gamma)$ and $Q 
= (\alpha', \beta', \gamma')$, then $P$ and $Q$ belong to the same line when, 
considered as points in~$\p^3_{\C}$, they are orthogonal with respect to the 
quadratic form determined by 
the equation of~$\SP_{\C}^2$, namely if:
\[
 \begin{pmatrix}
  \alpha & \beta & \gamma & 1
 \end{pmatrix}
 \begin{pmatrix}
  1 \\ 
  & 1 \\
  & & 1 \\
  & & & -1
 \end{pmatrix}
 \begin{pmatrix}
  \alpha' \\ \beta' \\ \gamma' \\ 1
 \end{pmatrix}
 = 0
 \quad \Leftrightarrow \quad 
 \left\langle
 \begin{pmatrix}
  \alpha \\ \beta \\ \gamma
 \end{pmatrix},
 \begin{pmatrix}
  \alpha' \\ \beta' \\ \gamma'
 \end{pmatrix}
 \right\rangle - 1 = 0 \,.
\]
Hence, if we suppose that $\vec{P}$ is a realization of a Laman graph~$G$ 
compatible with a general assignment of spherical distances for its edges, then 
$\vec{P}$ determines a $2n$-tuple of distinct points $(\vec{P^{l}}, 
\vec{P^{r}})$ on the absolute conic~$A$.
\end{remark}

The combination of Lemma~\ref{lemma:equivalence}, 
Proposition~\ref{proposition:translation}, and Remark~\ref{remark:distinct} 
shows that, instead of considering $n$-tuples $\vec{P}$ that are realizations of 
a Laman graph compatible with a general assignment of spherical distances for 
its edges, up to the action of~$\SO_3(\C)$, we can consider $2n$-tuples 
$(\vec{P^{l}}, \vec{P^{r}})$ of points on~$\p^1_{\C}$ for which some 
cross-ratios are assigned, up to the action of~$\pgl_2(\C)$. The latter are 
elements of a so-called \emph{moduli space} of curves of genus zero with marked 
points. In the next section we describe this object and its properties 
concerning intersection theory. Afterwards, we come back to our original 
problem and cast it into this theoretical framework.

\section{The moduli space of stable curves of genus zero with marked points}
\label{moduli_space}

In this section we describe for the reader's convenience the well-known moduli 
space of stable curves of genus zero with marked points and its Chow ring, 
following~\cite{Keel1992}. No new results are presented in this section.

Let us start by recalling a basic and fundamental result in projective 
geometry: every triple of distinct points~$P$, $Q$, and~$R$ in~$\p^1_{\C}$ can 
be mapped to the triple $(1:0)$, $(0,1)$, and~$(1:1)$ via a unique automorphism 
of~$\p^1_{\C}$, namely an element of~$\pgl_2(\C)$. Hence, any triple of distinct 
points in~$\p^1_{\C}$ is projectively equivalent to any other one, namely there 
always exists an automorphism mapping one to the other. If we consider a 
$4$-tuple of distinct points $P$, $Q$, $R$, $S$, then there is a unique element 
that takes it to the $4$-tuple $(1:0)$, $(0:1)$, $(1:1)$, and~$(1:\lambda)$: 
the number~$\lambda$ is called the \emph{cross-ratio} of the tuple $(P,Q,R,S)$. 
Two $4$-tuples of distinct points are then projectively equivalent if and only 
if their cross-ratios are the same. Hence, each equivalence class of $4$-tuples 
of distinct points modulo~$\pgl_2(\C)$ is represented by its cross-ratio, 
namely by an element in $\C \setminus \{ 0, 1 \}$, or equivalently in 
$\p^1_{\C} 
\setminus \{ (1:0), (0:1), (1:1) \}$. We then say that $\p^1_{\C} 
\setminus \{ (1:0), (0:1), (1:1) \}$ is the \emph{moduli space} of $4$-tuples 
of distinct points in~$\p^1_{\C}$. 

We can consider, for every $m \geq 4$, the space of equivalence classes of 
$m$-tuples of distinct points under the action of~$\pgl_2(\C)$: each such 
equivalence class is uniquely determined by an element in the quasi-projective 
variety
\[
  \bigl(
   \p^1_{\C} \setminus \{ (1:0), (0:1), (1:1) \}
  \bigr)^{m-3}
  \setminus
  \left\{
  \begin{array}{c}
   (m-3)\text{-tuples having at least} \\
   \text{two equal components}
  \end{array}
  \right\} \,.
\]
This space is called the \emph{moduli space} of $m$-tuples of distinct points 
in~$\p^1_{\C}$, and is denoted by~$\mscr{M}_{0,m}$. One may notice 
that this space is not compact under the Euclidean topology (in more 
algebro-geometric terms, it is not complete), and this may be a problem when 
dealing with enumerative questions, as the one we address in this work. 
Because of this, researchers focused on finding \emph{compactifications} of 
these moduli spaces. A possible smooth compactification of the 
space~$\mscr{M}_{0,m}$, denoted~$\M_{0,m}$, was constructed by 
Knudsen~\cite{Knudsen1983} (see \cite{Mumford1965, Mumford1977, Gieseker1982} 
for a more general account on the topic). This construction introduces a 
\emph{boundary} for~$\mscr{M}_{0,m}$, constituted of particular curves, called 
\emph{stable curves}, which are essentially reducible curves whose irreducible 
components are rational curves, intersecting in nodes. More precisely, we have:

\begin{definition}[{\cite[Introduction]{Keel1992}}]
 A \emph{stable curve of genus zero with $m$ marked points} is a reduced, 
possibly reducible, curve~$C$ with at worst node singularities, 
together with $m$ distinct marked points $p_1, \dotsc, p_m$ on it such that:
\begin{itemize}
 \item the points $\{ p_i \}_{i=1}^{m}$ lie on the smooth locus of~$C$;
 \item each irreducible component of~$C$ is isomorphic to~$\p^1_{\C}$, and 
altogether all irreducible components form a tree;
 \item for each irreducible component of~$C$, the sum of the numbers of 
singular points and of marked points on that component is at least~$3$.
\end{itemize}
\end{definition}

The geometry of~$\M_{0,m}$ is rich and well-studied: we refer 
to~\cite[Introduction]{Keel1992}, \cite{Kapranov1993}, 
and~\cite[Chapter~1]{Kock2007} for a discussion.

As we are going to see in Section~\ref{algorithm}, our algorithm relies on the 
understanding of how subvarieties of~$\M_{0,m}$ intersect each other. This 
piece of information is encoded in the so-called \emph{Chow ring}, which is a 
standard object in intersection theory. For its definition and properties we 
refer to the introduction~\cite{Fulton1984}, or to the standard 
book~\cite{Fulton1998}.
 
For the reader's convenience, we briefly provide in 
Theorem~\ref{theorem:chow_ring} the description by Keel of the Chow ring 
of~$\M_{0,m}$. First, we need to introduce some particular divisors that Keel 
calls ``vital''. 

\begin{definition}[{\cite[Introduction]{Keel1992}}]
\label{definition:vital_divisor}
 Let $(I,J)$ be a partition of~$\{1, \dotsc, m\}$ where $|I| \geq 2$ and $|J| 
\geq 2$. We define the divisor~$D_{I,J}$ in~$\M_{0,m}$ to be the divisor whose 
general point is a stable curve with two irreducible components such that the 
marked points labeled by~$I$ lie on one component, while the marked points 
labeled by~$J$ lie on the other component.
\end{definition}

\begin{proposition}[Knudsen, {\cite[Introduction and Fact~$2$]{Keel1992}}]
\label{proposition:vital_product}
 Any divisor $D_{I,J}$ as in Definition~\ref{definition:vital_divisor} is smooth 
and it is isomorphic to the product~$\M_{0,|I|+1} \times \M_{0, |J| + 1}$; in 
the previous isomorphism, the point of intersection of the two components of a
general stable curve in~$D_{I,J}$ counts as an extra marked point in each of 
the two factors of the product.
\end{proposition}

\begin{theorem}[{\cite[Introduction and Theorem~$1$]{Keel1992}}]
\label{theorem:chow_ring}
 The Chow ring of~$\M_{0,m}$ admits the following description:
\[
  \Z \bigl[
   D_{I,J} \,\colon\, 
   (I,J) \mathrm{\ is\ a\ partition\ of\ } \{1, \dotsc, m\} 
   \mathrm{\ where\ } |I| \geq 2 \mathrm{\ and\ } |J| \geq 2 
   \bigr] \, \Big/\sim\,,
\]
where the equivalence $\sim$ is given by the relations:
\begin{itemize}
 \item
  $D_{I,J} = D_{J,I}$,
 \item
  for any four distinct elements $a,b,c,d \in \{ 1, \dotsc, m \}$ we have
  \begin{equation}
  \tag{$\triangle$}
  \label{eq:relations}
   \sum_{\substack{a, b \in I \\ c, d \in J}} D_{I,J} = 
   \sum_{\substack{a, c \in I \\ b, d \in J}} D_{I,J} = 
   \sum_{\substack{a, d \in I \\ b, c \in J}} D_{I,J} \, ,
  \end{equation}
 \item
  $D_{I,J} \cdot D_{K,L} = 0$ unless one of the following holds:
  \[
   I \subseteq K, \mathrm{\ or\ } 
   K \subseteq I, \mathrm{\ or\ } 
   J \subseteq L, \mathrm{\ or\ } 
   L \subseteq J \,. 
  \]
\end{itemize}
Moreover, the three sums in Equation~\eqref{eq:relations} are the pullbacks 
of the respective divisors $D_{\{a,b\},\{c,d\}}$, $D_{\{a,c\},\{b,d\}}$, and 
$D_{\{a,d\},\{b,c\}}$ under the map 
\[
 \pi_{a,b,c,d} \colon \M_{0,m} \longrightarrow \M_{0,4} \cong \p^1_{\C}\,.
\]
Here, the map $\pi_{a,b,c,d}$ is the map 
that forgets all the marked points except for the ones labeled by $a$, $b$, 
$c$, and~$d$.
\end{theorem}

\section{The algorithm}
\label{algorithm}

At the end of Section~\ref{realizations} we understood that realizations on the 
sphere up to~$\SO_3(\C)$ can be considered as elements of a moduli space. In 
particular, configurations of $n$ points on the complex sphere up 
to~$\SO_3(\C)$ correspond to configurations of $2n$ points on a rational 
curve up to~$\pgl_2(\C)$, i.e., elements in the moduli space of rational 
curves with marked points. Moreover, assigning angles between two points on a 
sphere corresponds to assigning the cross-ratio of the $4$-tuple constituted of 
the left and right lifts of those two points. The elements of the moduli 
space, for which the cross-ratio of $4$ marked points is prescribed, are fibers 
of the map
\[
 \M_{0,2n} \longrightarrow \M_{0,4} \cong \p^1_{\C}
\]
that forgets all but the $4$ considered marked points (see 
Theorem~\ref{theorem:chow_ring}). Hence, the elements of the moduli space we 
are interested in, namely the ones for which the cross-ratios are specified for 
the $4$-tuples arising from edges of a graph, are fibers of products of these 
maps. We give a name to these maps:

\begin{definition}
\label{definition:map}
 Let $G = (V,E)$ be a graph, and suppose that $V = \{1, \dotsc, n\}$. We 
define the morphism
\[
 \Phi_{G} \colon 
 \M_{0,2n} \longrightarrow 
 \prod_{\{a,b\} \in E} \M_{0,4}^{a, b, a+n, b+n}
\]
whose components are the maps $\pi_{a,b,a+n,b+n} \colon \M_{0,2n} 
\longrightarrow \M_{0,4}$ forgetting all but $4$ marked points and defined in 
Theorem~\ref{theorem:chow_ring}. The choice for the indices reflects the 
labeling of the marked points on the rational curves by $(P_1^{l}, \dotsc, 
P_n^{l}, P_1^{r}, \dotsc, P_n^{r})$, putting first the points corresponding to 
left lifts, and then the points corresponding to right lifts of points on the 
sphere. The translation we operated in Section~\ref{realizations} tells us that 
$\Phi_G$ is a dominant morphism between smooth varieties of the same dimension, 
so its general fibers are constituted of finitely many points.
\end{definition}

\begin{remark}
 For any Laman graph~$G$ with $n$ vertices, the image of the boundary 
of~$\M_{0,2n}$ under~$\Phi_G$ is a proper subvariety of $\prod_{\{a,b\} \in E} 
\M_{0,4}^{a, b, a+n, b+n}$. This means that a general fiber of~$\Phi_G$ will not 
intersect the boundary, and so it is constituted of classes of rational curves 
with $2n$ distinct marked points. Each of such rational curves is then 
isomorphic to the absolute conic~$A$, and the $2n$ marked points determine a 
realization of~$G$ on the sphere. Moreover, when $G$ is a Laman graph, a general 
fiber of~$\Phi_G$ is a complete intersection in~$M_{0,2n}$ and it is constituted 
by reduced points. 
\end{remark}

The discussion so far proves the following theorem.

\begin{theorem}
\label{theorem:equivalence}
 Given a Laman graph $G$, then the number of realizations 
of~$G$ in~$S^2_{\C}$ for a general assignment of spherical distances for its 
edges, up to~$\SO_3(\C)$, equals the cardinality of a general fiber of the 
map~$\Phi_G$ as in Definition~\ref{definition:map}.
\end{theorem}

In the light of Theorem~\ref{theorem:equivalence}, our goal becomes, given a 
Laman graph~$G$, to compute the cardinality of a general fiber of~$\Phi_G$. 

\begin{remark}
\label{remark:fiber_class}
 Let us focus on how the class in the Chow ring of a fiber of~$\Phi_G$ looks 
like. Since the codomain of~$\Phi_G$ is a product, we can express a fiber as the 
intersection of the fibers of the maps to each component of the product. In 
other words, if we fix $\Lambda \in \prod_{\{a,b\} \in E} \M_{0,4}^{a, b, a+n, 
b+n}$ and we denote by~$\Phi_{a,b}$ the composition of~$\Phi_G$ with the 
projection to the factor~$\M_{0,4}^{a, b, a+n, b+n}$, then we have that
\[
 \Phi_G^{-1}(\Lambda) = \bigcap_{\{a,b\} \in E}
\Phi_{a,b}^{-1}(\Lambda_{a,b}) \,.
\]
It follows then, using the notation of Theorem~\ref{theorem:chow_ring}, that 
the class of~$\Phi_G^{-1}(\Lambda)$ in the Chow ring of~$\M_{0,2n}$ is given by
\[
 \prod_{\{a,b\} \in E} \sum_{\substack{a, b \in I \\ a+n, b+n \in J}} D_{I,J} \,.
\]
\end{remark}

Since a general fiber of~$\Phi_G$ is a reduced complete intersection 
in~$\M_{0,2n}$ when $G$ is a Laman graph, its cardinality is the degree of its 
Chow class. Hence we get:

\begin{proposition}
\label{proposition:degree}
 Given a Laman graph~$G=(V,E)$ with $V = \{1, \dotsc, n\}$, the number of 
realizations of~$G$ in~$S^2_{\C}$ for a general assignment of spherical 
distances for its edges, up to~$\SO_3(\C)$, equals
\[
 \deg 
 \lowerparen{6pt}{
  \prod_{\{a,b\} \in E} 
  \sum_{\substack{a, b \in I \\ a+n, b+n \in J}} 
  D_{I,J} 
 } \,.
\]
\end{proposition}

Algorithm \hyperref[alg:CountRealizations]{\texttt{CountRealizations}} computes 
the degree in Proposition~\ref{proposition:degree} using the description of 
vital divisors~$D_{I,J}$ provided by Proposition~\ref{proposition:vital_product} 
and the relations in the Chow ring stated in Theorem~\ref{theorem:chow_ring}. In 
fact, if we fix an edge~$\{a_0, b_0\} \in E$ of a Laman graph $G = (V, E)$, then 
the class we want to compute is
\begin{equation}
\label{equation:select}
 \lowerparen{6pt}{
  \sum_{\substack{a_0, b_0 \in I_0 \\ 
   a_0 + n, b_0 + n \in J_0}} D_{I_0, J_0}
 } \cdot
 \underbrace{
 \lowerparen{6pt}{
  \prod_{\{a,b\} \in E \setminus \{a_0, b_0\}} 
  \sum_{\substack{a, b \in I \\ a+n, b+n \in J}} D_{I,J} 
 }}_{=: F_{a_0, b_0}} \,.
\end{equation}
For every $(I_0, J_0)$ such that $a_0, b_0 \in I_0$ 
and $a_0 + n, b_0 + n \in J_0$, the product $D_{I_0, J_0} 
\cdot F_{a_0, b_0}$ can be computed by restricting~$F_{a_0, 
b_0}$ to~$D_{I_0, J_0}$ and using the isomorphism $D_{I_0, 
J_0} \cong \M_{0, I_0 \cup \{\ast\}} \times \M_{0,J_0 \cup 
\{\ast\}}$. We show that both the restrictions of~$F_{a_0, 
b_0}$ to~$\M_{0, I_0 \cup \{\ast\}}$ and to~$\M_{0, J_0 \cup 
\{\ast\}}$ have the same structure of the initial class we wanted to compute, 
and this determines a recursive procedure to solve our task.

To clarify the recursive procedure, let us start by noticing that the class
\begin{equation}
\label{equation:class}
  \prod_{\{a,b\} \in E} 
  \sum_{\substack{a, b \in I \\ a+n, b+n \in J}} 
  D_{I,J}
\end{equation}
from Proposition~\ref{proposition:degree} is a particular instance of a more 
general construction. The construction works as follows: we start from a 
set~$N$, and a set~$Q$ of $4$-tuples of distinct elements of~$N$. Then 
we form the following class in the Chow ring of~$\M_{0,|N|}$: 
\begin{equation}
\label{equation:class_combinatorial}
 \mcal{A}_{N,Q} := 
 \prod_{\substack{q \in Q \\ q = (a,b,c,d)}} \;
 \sum_{\substack{a, b \in I \\ c, d \in J}}
 D_{I,J} \,.
\end{equation}
Notice that if we set $N_0 := \{1, \dotsc, 2n\}$ and $Q_0 := 
\mbox{$\{ (a,b,a+n,b+n) \, \colon \, \{a, b\} \in E\}$}$, 
then the class from Equation~\eqref{equation:class} equals $\mcal{A}_{N_0, 
Q_0}$. In this perspective, also the class~$F_{a_0, b_0}$ from 
Equation~\eqref{equation:select} can be seen as a particular case of a general 
construction: what we do here is to select some $\bar{q} \in Q$, with $\bar{q} = 
(\bar{a}, \bar{b}, \bar{c}, \bar{d})$, and to define the class 
\[
 \mcal{G}_{\bar{q}} := 
 \prod_{\substack{q \in Q \setminus \{ \bar{q} \} \\ q = (a,b,c,d)}} \;
 \sum_{\substack{a, b \in I \\ c, d \in J}} D_{I,J} \,,
\]
so that we get the factorization
\begin{equation}
\label{equation:general_split}
 \mcal{A}_{N,Q} =  
 \lowerparen{6pt}{
  \sum_{\substack{\bar{a}, \bar{b} \in \bar{I} \\ \bar{c}, \bar{d} \in J}} D_{\bar{I},\bar{J}} 
 } \cdot \mcal{G}_{\bar{q}} \,.
\end{equation}
With this notation, the class $F_{a_0, b_0}$ equals $\mcal{G}_{q_0}$, where $q_0 = (a_0, b_0, a_0 + n, b_0 + n)$.

We show now that we can set up an iterative procedure for the computation of the 
degree of classes of type~$\mcal{A}_{N,Q}$. Notice that, taking into account 
Equation~\eqref{equation:general_split}, this can be achieved once we are able 
to compute the degree of a product $D_{\bar{I}, \bar{J}} \cdot 
\mcal{G}_{\bar{q}}$. We then fix $\bar{q} \in Q$ with $\bar{q} = (\bar{a}, 
\bar{b}, \bar{c}, \bar{d})$ and we select a pair $(\bar{I}, \bar{J})$ such that 
$\bar{a}, \bar{b} \in \bar{I}$ and $\bar{c}, \bar{d} \in \bar{J}$. If there 
exists $q \in Q \setminus \{\bar{q}\}$ such that $|q \cap \bar{I}| = |q \cap 
\bar{J}| = 2$, then the restriction of~$\mcal{G}_{\bar{q}}$ to $D_{\bar{I}, 
\bar{J}}$ is zero by \cite[Fact 2]{Keel1992}. Otherwise, the restriction 
of~$\mcal{G}_{\bar{q}}$ to $D_{\bar{I}, \bar{J}} \cong \M_{0, \bar{I} \cup 
\{\ast\}} \times \M_{0, \bar{J} \cup \{\ast\}}$ is the product of two 
classes~$\mcal{G}_{\bar{q}}^{\bar{I}}$ and~$\mcal{G}_{\bar{q}}^{\bar{J}}$. 
Recall, in fact, that the Chow ring of $\M_{0, \bar{I} \cup \{\ast\}} \times 
\M_{0, \bar{J} \cup \{\ast\}}$ is the tensor product of the Chow rings 
of~$\M_{0, \bar{I} \cup \{\ast\}}$ and~$\M_{0, \bar{J} \cup \{\ast\}}$ by 
\cite[Theorem~2]{Keel1992}. Analyzing the isomorphism making~$D_{\bar{I}, 
\bar{J}}$ into a product (see \cite[Fact~2]{Keel1992} and 
\cite[Theorem~3.7]{Knudsen1983}), one sees that the two 
classes~$\mcal{G}_{\bar{q}}^{\bar{I}}$ and~$\mcal{G}_{\bar{q}}^{\bar{J}}$ admit 
the following description. For $k \in \{ 0, \dotsc, 4 \}$, define the sets:
\[
 Q_{k, 4-k} :=
 \bigl\{
  q \in Q \setminus \{ \bar{q} \} 
  \, \colon \, 
  |q \cap \bar{I}| = k
 \bigr\}
 =
 \bigl\{
  q \in Q \setminus \{ \bar{q} \} 
  \, \colon \, 
  |q \cap \bar{J}| = 4-k
 \bigr\} \,.
\]
Notice that, by definition, all tuples in $Q_{3,1}$ have exactly one element 
in~$\bar{J}$. Define $Q_{3,1}'$ to be set obtained by substituting in all 
$4$-tuples of~$Q_{3,1}$ their element in~$\bar{J}$ by the new symbol~$\ast$. 
Analogously, define~$Q_{1,3}'$. Then the classes~$\mcal{G}_{\bar{q}}^{\bar{I}}$ 
and~$\mcal{G}_{\bar{q}}^{\bar{J}}$ are
\[
 \prod_{\substack{q \in Q_{4,0} \cup Q_{3,1}' \\ q=(a,b,c,d)}} \;
 \sum_{\substack{a,b\in I \\ c,d \in J}} D_{I,J}
 \quad \text{and} \quad
 \prod_{\substack{q \in Q_{0,4} \cup Q_{1,3}' \\ q=(a,b,c,d)}} \;
 \sum_{\substack{a,b\in I \\ c,d \in J}} D_{I,J} \,,
\]
where the $D_{I,J}$ are divisors in the appropriate moduli spaces, namely for 
$\mcal{G}_{\bar{q}}^{\bar{I}}$ they are divisors in~$\M_{0,|\bar{I} \cup 
\{\ast\}|}$, while for $\mcal{G}_{\bar{q}}^{\bar{J}}$ they are divisors 
in~$\M_{0,|\bar{J} \cup \{\ast\}|}$.

Hence, in the notation of Equation~\eqref{equation:class_combinatorial}, if we set 
\begin{align*}
 N_{\bar{I}} &:= \bar{I} \cup \{\ast\},  & Q_{\bar{I}} &:=  Q_{4,0} \cup Q_{3,1}',\\
 N_{\bar{J}} &:= \bar{J} \cup \{\ast\},  & Q_{\bar{J}} &:=  Q_{0,4} \cup Q_{1,3}',
\end{align*}
we have
\[
 \mcal{G}_{\bar{q}}^{\bar{I}} = \mcal{A}_{N_{\bar{I}}, Q_{\bar{I}}}
 \quad \text{and} \quad
 \mcal{G}_{\bar{q}}^{\bar{J}} = \mcal{A}_{N_{\bar{J}}, Q_{\bar{J}}} \,.
\]
Therefore, the degree of a class~$\mcal{A}_{N,Q}$ can be computed by fixing an 
element $\bar{q} \in Q$, $\bar{q} = (\bar{a}, \bar{b}, \bar{c}, \bar{d})$, and 
using Equation~\eqref{equation:general_split}, thus obtaining the formula 
\[
 \deg \mcal{A}_{N,Q} =  
  \sum_{\substack{\bar{a}, \bar{b} \in \bar{I}\\ \bar{c}, \bar{d} \in \bar{J}}} 
  \left(
  \deg \mcal{A}_{N_{\bar{I}}, Q_{\bar{I}}} \cdot 
  \deg \mcal{A}_{N_{\bar{J}}, Q_{\bar{J}}} 
  \right) \,,
\]
where it is intended that a summand is zero if the corresponding set $Q_{2,2}$ 
is not empty. This allows one to set up a recursive procedure for the 
computation of the degree of a class~$\mcal{A}_{N, Q}$---so in particular of the 
class of a fiber of~$\Phi_G$. The recursion stops if we reach one of these 
situations:
\begin{itemize}
 \item The set $Q_{2,2}$ is not empty: in this case we can skip the 
contribution given by this class, since its degree is zero.
 \item The set $N$ is composed of four elements, and $Q$ consists of a single 
tuple: in this case the degree of the class is~$1$.
 \item The cardinality of $Q_{4,0} \cup Q_{3,1}'$ is different from $|\bar{I} 
\cup \{\ast\}| - 3$ or the cardinality of~$Q_{0,4} \cup Q_{1,3}'$ is different 
from $|\bar{J} \cup \{\ast\}| - 3$: in this case either $\mcal{G}_{\bar{q}}^{\bar{I}}$ or $\mcal{G}_{\bar{q}}^{\bar{J}}$ is zero, and so this contribution can be skipped.
\end{itemize}
The discussion so far proves the correctness of Algorithm 
\hyperref[alg:CountRealizations]{\texttt{CountRealizations}}.
Termination is implied by the fact that the size of the sets always decreases and therefore, the base cases are reached.

\renewcommand{\thealgorithm}{} % Removes algorithm numbering
\begin{algorithm}
\caption{\texttt{CountRealizations}}\label{alg:CountRealizations}
\begin{algorithmic}[1]
  \Require A pair $(N, Q)$, where $N$ is a set and $Q$ is a list of $4$-tuples 
of elements of~ $N$.
  \Ensure A natural number. When, for a Laman graph $G=(V,E)$, 
with $V=\range{1}{n}$, we have $N=\range{1}{2n}$ and $Q=\{ 
(a,b,a+n,b+n) \, \colon \, \{a,b\}\in E\}$, then this natural number 
represents the number of realizations of~$G$ on the complex sphere, up 
to~$\SO_3(\C)$.
  \Statex
  \If{($|N| = 4$ and $|Q| = 1$) or ($|N| = 3$ and $|Q| = 0$)}
    \State \Return 1
  \EndIf
  \State {\bfseries Select} any element $\bar{q} \in Q$ and {\bfseries write} 
$\bar{q} = (\bar{a},\bar{b},\bar{c},\bar{d})$.
  \State {\bfseries Set} $Q' := Q \setminus \{\bar{q}\}$ and  $N' := N 
\setminus \{\bar{a},\bar{b},\bar{c},\bar{d}\}$.
  \State {\bfseries Compute} $\mscr{L} := \{ \mathrm{subsets\ of\ } N' \}$.
  \State {\bfseries Set} $S := \emptyset$.
  \For{each subset $L \in \mscr{L}$}
    \State {\bfseries Set} $\bar{I} := \{\bar{a},\bar{b}\} \cup L$ and $\bar{J} := \mathtt{complement\ of\ } \bar{I} \mathtt{\ in\ } N$. 
    \State {\bfseries Append} $(\bar{I},\bar{J})$ to $S$. 
  \EndFor
  \State {\bfseries Set} $\mathtt{sum} := 0$.
  \For{each pair $(\bar{I}, \bar{J})$ in $S$}
    \State {\bfseries Compute} the following five lists:
    \begin{gather*}
     Q_{4,0} := 
     \bigl\{ 
      q \in Q' \, \colon \, | q \cap \bar{I} | = 4 
     \bigr\} = 
     \bigl\{
      q \in Q' \, \colon \, |q \cap \bar{J} | = 0 
     \bigr\} \, , \\
     Q_{3,1} := 
     \bigl\{ 
      q \in Q' \, \colon \, | q \cap \bar{I} | = 3 
     \bigr\} = 
     \bigl\{
      q \in Q' \, \colon \, |q \cap \bar{J} | = 1 
     \bigr\} \, , \\
     Q_{2,2} := 
     \bigl\{ 
      q \in Q' \, \colon \, | q \cap \bar{I} | = 2 
     \bigr\} = 
     \bigl\{
      q \in Q' \, \colon \, |q \cap \bar{J} | = 2 
     \bigr\} \, , \\
     Q_{1,3} := 
     \bigl\{ 
      q \in Q' \, \colon \, | q \cap \bar{I} | = 1
     \bigr\} = 
     \bigl\{
      q \in Q' \, \colon \, |q \cap \bar{J} | = 3 
     \bigr\} \, , \\
     Q_{0,4} := 
     \bigl\{ 
      q \in Q' \, \colon \, | q \cap \bar{I} | = 0 
     \bigr\} = 
     \bigl\{
      q \in Q' \, \colon \, |q \cap \bar{J} | = 4
     \bigr\} \, .
    \end{gather*}
    \If{$|Q_{2,2}| > 0$}
      \State {\textbf{Continue}}
    \EndIf
    \State {\bfseries Let} $\ast$ be a new symbol, not belonging to~$N$.
    \State {\bfseries Set} $Q'_{3,1} := \emptyset$ and $Q'_{1,3} := \emptyset$.
    \For{each element $q \in Q_{3,1}$}
     \State {\bfseries Substitute} in $q$ the element $q \cap \bar{J}$ with $\ast$.
     \State {\bfseries Append} the resulting tuple to $Q'_{3,1}$.
    \EndFor
    \State {\bfseries Apply} the analogous procedure to the elements of $Q_{1,3}$, 
obtaining $Q'_{1,3}$.
    \If{$|Q_{4,0} \cup Q_{3,1}'| \neq |\bar{I} \cup \{\ast\}| - 3$ or $|Q_{4,0} \cup Q_{1,3}'| \neq | \bar{J} \cup \{\ast\}| - 3$}
      \State {\textbf{Continue}}
    \EndIf
    \State {\bfseries Update} (here \texttt{CR} stands for 
\texttt{CountRealizations})
    \[
    \mathtt{sum} := \mathtt{sum} + \mathtt{CR} \bigl( \bar{I} \cup \{ \ast \} \,,\, 
Q_{4,0} \cup Q'_{3,1} \bigr) \cdot \mathtt{CR} \bigl( \bar{J} \cup \{ \ast \} \,,\, 
Q_{0,4} \cup Q'_{1,3} \bigr)
    \]
  \EndFor
  \State \Return \texttt{sum}.
\end{algorithmic}
\end{algorithm}

\section{Computed Data}
\label{data}

Using Algorithm~\hyperref[alg:CountRealizations]{\texttt{CountRealizations}} we 
computed the number of realizations on the sphere of all Laman graphs with up 
to~$10$ vertices. Table~\ref{tab:max} lists those graphs that have the maximal 
number of realizations on the sphere within the class of graphs with the same 
number of vertices.

\begin{table}[ht]
  \begin{center}
		\caption{Graphs with maximal number of complex realizations on the sphere within 
graphs of $n$ vertices. $\Lam_2$ denotes the number of complex realizations in the 
plane.}
    \begin{tabular}{m{0.9cm}m{7cm}cc}
			\toprule
      n & Graph(s) & \#realizations & $\Lam_2$\\
      \midrule
      \addlinespace[10pt]
      5 & 
      \begin{tikzpicture}[scale=0.8]
				\node[vertex] (1) at (0.00434626, 0.) {};
				\node[vertex] (2) at (0., 1.46946) {};
				\node[vertex] (3) at (1.71639, 0.737365) {};
				\node[vertex] (4) at (0.708363, 0.451834) {};
				\node[vertex] (5) at (0.708156, 1.02237) {};
				\draw[edge] (1)edge(4) (1)edge(5) (2)edge(4) (2)edge(5) (3)edge(4) (3)edge(5) (4)edge(5);
			\end{tikzpicture}
			\quad
			\begin{tikzpicture}[scale=0.8]
				\node[vertex] (1) at (2.19943, 0.770596) {};
				\node[vertex] (2) at (0., 0.770235) {};
				\node[vertex] (3) at (0.572073, 0.) {};
				\node[vertex] (4) at (1.62787, 0.000163602) {};
				\node[vertex] (5) at (1.09881, 0.729856) {};
				\draw[edge] (1)edge(4) (1)edge(5) (2)edge(3) (2)edge(5) (3)edge(4) (3)edge(5) (4)edge(5);
			\end{tikzpicture}
			\quad
			\begin{tikzpicture}[scale=0.8]
				\node[vertex] (1) at (0., 0.489214) {};
				\node[vertex] (2) at (1.77414, 0.846507) {};
				\node[vertex] (3) at (1.77498, 0.132723) {};
				\node[vertex] (4) at (0.943388, 0.) {};
				\node[vertex] (5) at (0.942996, 0.978813) {};
				\draw[edge] (1)edge(4) (1)edge(5) (2)edge(3) (2)edge(4) (2)edge(5) (3)edge(4) (3)edge(5);
			\end{tikzpicture}
			 & 8 & 8|8|8\\[6ex]
      6 & 
      \begin{tikzpicture}[scale=0.8]
				\node[vertex] (1) at (0., 1.11908) {};
				\node[vertex] (2) at (1.71676, 0.586005) {};
				\node[vertex] (3) at (1.13954, 0.) {};
				\node[vertex] (4) at (0.0000860155, 0.0542901) {};
				\node[vertex] (5) at (0.56816, 0.586542) {};
				\node[vertex] (6) at (1.14092, 1.17299) {};
				\draw[edge] (1)edge(4) (1)edge(5) (1)edge(6) (2)edge(3) (2)edge(5) (2)edge(6) (3)edge(4) (3)edge(6) (4)edge(5);
			\end{tikzpicture}
      & 32 & 24\\[6ex]
      \multirow{2}{*}{7} & 
      \begin{tikzpicture}[scale=0.8]
				\node[vertex] (1) at (0., 0.902612) {};
				\node[vertex] (2) at (2.0773, 1.40564) {};
				\node[vertex] (3) at (2.07508, 0.394714) {};
				\node[vertex] (4) at (1.33654, 0.) {};
				\node[vertex] (5) at (1.33847, 1.80086) {};
				\node[vertex] (6) at (0.998751, 0.62641) {};
				\node[vertex] (7) at (1.00049, 1.17512) {};
				\draw[edge] (1)edge(6) (1)edge(7) (2)edge(3) (2)edge(5) (2)edge(7) (3)edge(4) (3)edge(6) (4)edge(6) (4)edge(7) (5)edge(6) (5)edge(7);
			\end{tikzpicture}
			\quad
			\begin{tikzpicture}[scale=0.8]
				\node[vertex] (1) at (0., 0.993319) {};
				\node[vertex] (2) at (1.65317, 0.651298) {};
				\node[vertex] (3) at (2.21524, 1.06624) {};
				\node[vertex] (4) at (1.88256, 0.) {};
				\node[vertex] (5) at (1.85979, 1.72831) {};
				\node[vertex] (6) at (1.02017, 1.34532) {};
				\node[vertex] (7) at (0.827394, 0.399951) {};
				\draw[edge] (1)edge(6) (1)edge(7) (2)edge(4) (2)edge(5) (2)edge(7) (3)edge(4) (3)edge(5) (3)edge(6) (4)edge(7) (5)edge(6) (6)edge(7);
			\end{tikzpicture}
			\quad
			\begin{tikzpicture}[scale=0.8]
				\node[vertex] (1) at (0., 0.840839) {};
				\node[vertex] (2) at (2.67089, 0.465001) {};
				\node[vertex] (3) at (2.19637, 1.38739) {};
				\node[vertex] (4) at (1.96024, 0.696549) {};
				\node[vertex] (5) at (1.64251, 0.) {};
				\node[vertex] (6) at (0.841646, 0.324386) {};
				\node[vertex] (7) at (1.05564, 1.00324) {};
				\draw[edge] (1)edge(6) (1)edge(7) (2)edge(3) (2)edge(4) (2)edge(5) (3)edge(4) (3)edge(7) (4)edge(6) (5)edge(6) (5)edge(7) (6)edge(7);
			\end{tikzpicture}
			& \multirow{2}{*}{64} & 48|48|48|\\[-1ex]
			&
			\begin{tikzpicture}[scale=0.8,baseline=1cm]
				\node[vertex] (1) at (0.91642, 0.) {};
				\node[vertex] (2) at (0.914886, 1.29661) {};
				\node[vertex] (3) at (0., 1.10126) {};
				\node[vertex] (4) at (1.83015, 1.10184) {};
				\node[vertex] (5) at (1.83103, 0.195186) {};
				\node[vertex] (6) at (0.000920175, 0.193779) {};
				\node[vertex] (7) at (0.914872, 0.647782) {};
				\draw[edge] (1)edge(2) (1)edge(5) (1)edge(6) (2)edge(3) (2)edge(4) (3)edge(6) (3)edge(7) (4)edge(5) (4)edge(7) (5)edge(7) (6)edge(7);
			\end{tikzpicture}
			\quad
			\begin{tikzpicture}[scale=0.8,baseline=1cm]
				\node[vertex] (1) at (1.72686, 0.227826) {};
				\node[vertex] (2) at (1.72504, 1.43682) {};
				\node[vertex] (3) at (0., 0.830988) {};
				\node[vertex] (4) at (0.624184, 1.66309) {};
				\node[vertex] (5) at (0.62616, 0.) {};
				\node[vertex] (6) at (2.11526, 0.832802) {};
				\node[vertex] (7) at (0.880502, 0.832071) {};
				\draw[edge] (1)edge(2) (1)edge(5) (1)edge(6) (2)edge(4) (2)edge(6) (3)edge(4) (3)edge(5) (3)edge(7) (4)edge(7) (5)edge(7) (6)edge(7);
			\end{tikzpicture}
      &&56|48\\[6ex]
      8 & 
      \begin{tikzpicture}[scale=0.8]
				\node[vertex] (1) at (1.64893, 1.43348) {};
				\node[vertex] (2) at (0.766508, 1.43208) {};
				\node[vertex] (3) at (0., 0.796979) {};
				\node[vertex] (4) at (2.41762, 0.798665) {};
				\node[vertex] (5) at (2.02749, 0.000384359) {};
				\node[vertex] (6) at (0.391446, 0.) {};
				\node[vertex] (7) at (1.47745, 0.497343) {};
				\node[vertex] (8) at (0.94273, 0.497205) {};
				\draw[edge] (1)edge(2) (1)edge(4) (1)edge(8) (2)edge(3) (2)edge(7) (3)edge(6) (3)edge(8) (4)edge(5) (4)edge(7) (5)edge(7) (5)edge(8) (6)edge(7) (6)edge(8);
			\end{tikzpicture}
			\quad
			\begin{tikzpicture}[scale=0.8]
				\node[vertex] (1) at (0.506062, 1.63004) {};
				\node[vertex] (2) at (1.51265, 1.63163) {};
				\node[vertex] (3) at (2.02315, 0.705683) {};
				\node[vertex] (4) at (0., 0.70201) {};
				\node[vertex] (5) at (0.208568, 0.) {};
				\node[vertex] (6) at (1.81708, 0.00285166) {};
				\node[vertex] (7) at (1.01292, 0.0412725) {};
				\node[vertex] (8) at (1.01131, 0.838488) {};
				\draw[edge] (1)edge(2) (1)edge(4) (1)edge(8) (2)edge(3) (2)edge(8) (3)edge(6) (3)edge(7) (4)edge(5) (4)edge(7) (5)edge(7) (5)edge(8) (6)edge(7) (6)edge(8);
			\end{tikzpicture}
      & 192 & 136|128\\[6ex]
      9 & 
      \begin{tikzpicture}[scale=0.8]
				\node[vertex] (1) at (0.458732, 1.92024) {};
				\node[vertex] (2) at (1.7521, 1.92592) {};
				\node[vertex] (3) at (0.65336, 0.) {};
				\node[vertex] (4) at (1.57651, 0.00626765) {};
				\node[vertex] (5) at (2.21727, 1.12881) {};
				\node[vertex] (6) at (0., 1.11783) {};
				\node[vertex] (7) at (1.11128, 0.518085) {};
				\node[vertex] (8) at (1.54049, 1.26802) {};
				\node[vertex] (9) at (0.675873, 1.26316) {};
				\draw[edge] (1)edge(6) (1)edge(8) (1)edge(9) (2)edge(5) (2)edge(8) (2)edge(9) (3)edge(4) (3)edge(7) (3)edge(9) (4)edge(7) (4)edge(8) (5)edge(7) (5)edge(8) (6)edge(7) (6)edge(9);
			\end{tikzpicture}
      & 576 & 320\\
      \bottomrule
    \end{tabular}
    \label{tab:max}
  \end{center}
\end{table}

\begin{remark}
 The paper \cite{Bartzos2018} shows that the number of real spherical 
realizations matches the number of complex ones for some graphs in 
Table~\ref{tab:max} (all graphs with $6$ and $7$ vertices, and one of the graphs 
with 8 vertices).
\end{remark}

Note that up till 8 vertices the graph with maximal Laman number 
(i.e., number of realizations in the plane) is also in the list of graphs with 
maximal number of realizations on the sphere. However, the graph with maximal 
Laman number with 9 vertices is different from the one with maximal number 
of realizations on the sphere. The latter has a very particular structure (see last row of
Table~\ref{tab:max}).

Our recursive algorithm gives a significant improvement over the naive
approach, which is to determine the number of solutions via a Gr\"obner basis computation.
Furthermore, the Gr\"obner basis approach needs randomly fixed edge lengths where \hyperref[alg:CountRealizations]{\texttt{CountRealizations}}
computes the numbers symbolically.
% For a graph with 8 vertices and maximal number of realizations (see Table~\ref{tab:max}) our algorithm needs 1.15s in Mathematica and 0.83s in Python.
% The Gr\"obner basis computation needed 120s in Mathematica and 8s in Maple.
For a graph with 9 vertices and maximal number of realizations (see Table~\ref{tab:max}) our algorithm needs 5.66s in Mathematica and 3.57s in Python.
The Gr\"obner basis computation needed 5850s in Mathematica and 27s in Maple.

\newcommand{\etalchar}[1]{$^{#1}$}
\providecommand{\bysame}{\leavevmode\hbox to3em{\hrulefill}\thinspace}
\providecommand{\MR}{\relax\ifhmode\unskip\space\fi MR }
\providecommand{\MRhref}[2]{%
  \href{http://www.ams.org/mathscinet-getitem?mr=#1}{#2}
}
\providecommand{\href}[2]{#2}

\vspace{-0.2cm}
\end{document}